\documentclass[12pt]{amsart}
\usepackage{amsmath, amscd}
\usepackage{amsfonts}
\usepackage{version}
\usepackage{amssymb}

\newtheorem{theorem}{Theorem}[section]
\newtheorem{lemma}[theorem]{Lemma}
\newtheorem{corollary}[theorem]{Corollary}

\newtheorem{question}[theorem]{Question}

\theoremstyle{definition}
\newtheorem{definition}[theorem]{Definition}

\theoremstyle{remark}
\newtheorem{remark}[theorem]{Remark}

\numberwithin{equation}{section}

\makeatletter
\@namedef{subjclassname@2020}{%
  \textup{2020} Mathematics Subject Classification}
\makeatother

\newcommand{\B}{\mathbb{B}}

\newcommand{\C}{\mathbb{C}}

\newcommand{\U}{\mathbb{U}}
\newcommand{\N}{\mathbb{N}}

\begin{document}

\title[Bieberbach conjecture, Bohr radius, and Bloch constant]
{Bieberbach conjecture, Bohr radius, Bloch constant and Alexander's theorem
%for a subclass of holomorphic mappings
in infinite dimensions
}

%\dedicatory{Dedicated to the memory of Gabriela Kohr.}     % REMOVE THIS LINE IF NOT NECESSARY

\author[H. Hamada]{Hidetaka Hamada$^{\ast}$}
\thanks{$^{\ast}$ Corresponding author}
\address{%
Faculty of Science and Engineering\\
Kyushu Sangyo University\\
3-1 Matsukadai, 2-Chome, Higashi-ku\\
Fukuoka 813-8503\\
Japan}
\email{h.hamada@ip.kyusan-u.ac.jp}

\author[G. Kohr]{Gabriela Kohr}
\address{Faculty of Mathematics and Computer Science,
Babe\c{s}-Bolyai University, 1 M. Kog\u{a}l\-niceanu Str., 400084,
Cluj-Napoca, Romania}
\email{gkohr@math.ubbcluj.ro}

\author[M. Kohr]{Mirela Kohr}
\address{Faculty of Mathematics and Computer Science,
Babe\c{s}-Bolyai University, 1 M. Kog\u{a}l\-niceanu Str., 400084,
Cluj-Napoca, Romania}
\email{mkohr@math.ubbcluj.ro}

%\thanks{H. Hamada was partially supported by JSPS KAKENHI Grant Number JP19K03553.}


\date{\today}


\begin{abstract}
In this paper, we investigate holomorphic mappings $F$ on the unit ball $\mathbb{B}$ of a complex Banach space of the form $F(x)=f(x)x$,
where $f$ is a holomorphic function on $\mathbb{B}$.
First, we investigate 
criteria for univalence, starlikeness and quasi-convexity of type $B$ on $\B$.
Next, we investigate a generalized Bieberbach conjecture,
a covering theorem and a distortion theorem,
the Fekete-Szeg\"{o} inequality,
lower bound for the Bloch constant,
and
Alexander's type theorem
for such mappings.
\end{abstract}

\maketitle

\section{Introduction}
\label{intro}
\setcounter{equation}{0}
Let $\mathbb{C}^{n}$  be the complex space of dimension $n$, where $n$ is a positive integer.
In particular, let  $\mathbb{C}:=\mathbb{C}^{1}$ be the complex plane. Let $\mathbb N$ denote the set of positive integers.
In one complex variable,
de Branges \cite{dB85}
proved the following Bieberbach conjecture
\cite{B16}.
If $f$ is a normalized univalent holomorphic function on the unit disc $\U$ in $\C$
with the Taylor expansion
\[
f(\zeta)=\zeta+\sum_{k=2}^{\infty}a_k\zeta^k,
\]
then the sharp estimates $|a_k|\leq k$ holds for $k\in \N$.
In several complex variables,
it is known that the Bieberbach conjecture does not hold in general.
For example, for $a\in \C$, the mapping
\begin{equation}
\label{example-1}
F_{a}(z_1,z_2)=(z_1+az_2^2, z_2),
\quad
(z_1,z_2)\in \C^2
\end{equation}
is normalized biholomorphic on $\C^2$, but
does not satisfy the Bieberbach conjecture with respect to an arbitrary norm on $\C^2$
(see also Cartan \cite{Cartan33}).
So, to hold the Bieberbach conjecture in several complex variables,
we need to assume some restrictions on the structure of the biholomorphic mappings.
Let $\B$ be the unit ball of a complex Banach space $X$.
For biholomorphic mappings of the form $F(x)=f(x)x$, $x\in \B$,
with $f(0)=1$,
partial results about the  Bieberbach conjecture on $\B$
have been obtained by adding some analytic or geometric conditions
on the mappings $F$
(\cite{LL16, LL17, LL18, LL21,  LLX15, XSO23, XXZ24}).
So, the following natural question arises.

\begin{question}
\label{Bieberbach-conjecture}
Let $\B$ be the unit ball of a complex Banach space $X$
and let $F:\B\to X$ be a biholomorphic mapping of the form $F(x)=f(x)x$, $x\in \B$,with $f(0)=1$.
Does the Bieberbach conjecture hold for such $F$
without adding any analytic or geometric condition
on the mapping $F$?
\end{question}

In this paper, we give a positive answer to this question by using a new method of proof.
We first prove the criteria for
univalence, starlikeness and quasi-convexity
which will be used in many places of this paper.
Next, by applying the univalence criterion (Theorem \ref{thm-univalence}),
we will prove the generalized Bieberbach conjecture for all holomorphic mappings $G$ on $\B$
of the form $G(x)=g(x)x$ which are subordinate to a biholomorphic mapping $F$ on $\B$ of the form $F(x)=f(x)x$.
In particular, if $G=F$ and $G$ is normalized, then we obtain a positive answer to Question \ref{Bieberbach-conjecture}.
As a corollary of the Bieberbach conjecture,
we obtain a sharp distortion theorem for normalized biholomorphic mappings $F$ on the unit ball of a complex Hilbert space
of the form $F(x)=f(x)x$.

A generalization of  the classical Fekete-Szeg\"{o} inequality
for normalized univalent functions $f$ on the unit disc $\U$
to univalent mappings in several complex variables
was first obtained by the authors in \cite{HKK2022JMAA}, \cite{HKK2022Mathematics}
by using the Loewner theory in several complex variables.
On the other hand, Fekete-Szeg\"{o} inequalities
for normalized biholomorphic mappings $F$ of the form $F(x)=f(x)x$
have been given
by adding some geometric or analytic properties such as
starlike mappings \cite{XLL18}
close-to-starlike mappings \cite{XFFL23},
spirallike mappings \cite{EJ22}, \cite{LX21}.
In this paper,  by applying the univalence criterion (Theorem \ref{thm-univalence}),
we give a generalization of  the classical Fekete-Szeg\"{o} inequality
for normalized univalent functions $f$ on the unit disc $\U$
to normalized biholomorphic mappings $F$ on $\B$
of the form $F(x)=f(x)x$
without using the Loewner theory in several complex variables
and without any geometric or analytic restriction.

Kayumov et al. \cite[Theorems 5 and 6]{KKP21}
gave the Bohr-Rogosinski radius
for holomorphic mappings $g$ on $\U$ which are subordinate to
$f$,
where $f$ is a univalent function on $\U$ or a convex (univalent) function on $\U$.
Note that
\cite[Theorem 6]{KKP21} (the case when $f$ is a convex function)
contains classical Bohr's inequality for holomorphic functions in the unit disc.
Hamada et al. \cite{HHK-new}
generalized \cite[Theorem 6]{KKP21} to holomorphic mappings $G$ on the unit ball $\B$ of a complex Banach space $X$
which are subordinate to a convex (biholomorphic) mappings on $\B$
and also gave a counterexample which shows that Theorem 5 in \cite{KKP21}
cannot be generalized to holomorphic mappings on the unit ball of a complex Banach space.
In this paper, as a corollary of the generalized Bieberbach conjecture,
we give a generalization of \cite[Theorem 5]{KKP21}
to holomorphic mappings $G$ of the form $G(x)=a+g(x)x$
on the unit ball $\B$ of a complex Banach space $X$
which are subordinate to a biholomorphic mapping on $\B$ of the form
$F(x)=a+f(x)x$.

Let $f:\U\rightarrow \mathbb{C}$ be a holomorphic function with $f'(0)=1$.
The celebrated Bloch's theorem states that $f$  maps a domain in $\U$ biholomorphically onto a disc with radius $r(f)$ greater than some positive absolute constant.
The `best possible' constant $\mathbf{B}$ for all such functions, that is,
$$\mathbf{B}=\inf\{r(f): f \mbox{ is holomorphic on $\U$
and $f'(0)=1$}\},$$  is called the Bloch constant.
Bonk \cite{B90} proved a distortion theorem for
holomorphic functions $f$ such that $f'(0)=1$ and
\[
\sup_{\zeta\in \U}(1-|\zeta|^2)|f'(\zeta)|\leq 1
\]
and applied it to prove a result of Ahlfors \cite{a}, which asserts that
the Bloch constant $\mathbf{B}$ is greater than ${\sqrt{3}}/{4}$.
This lower bound
was further improved in \cite{B90} to $\mathbf{B}>\frac{\sqrt{3}}{4}+10^{-14}$, and in \cite{CG96}
to $\mathbf{B}\geq \frac{\sqrt{3}}{4}+2\times 10^{-4}$
(cf.
\cite{B90},
\cite{BMY97},
\cite{LM92}).
In the case of several complex variables,
Liu \cite{L92} proved the Bonk's distortion theorem
for Bloch mappings on the Euclidean unit ball $\mathbb{B}^n$ of $\mathbb{C}^n$
into $\mathbb{C}^n$
and applied it to obtain a lower bound for the Bloch constant.
Wang and Liu \cite{WL08} proved the Bonk's distortion theorem
for Bloch mappings on the unit polydisc $\U^n$ of $\mathbb{C}^n$ into $\mathbb{C}^n$
and applied it to obtain a lower bound for the Bloch constant.
The first author \cite{H19JAM}
generalized the above results to Bloch mappings
on the unit ball of a finite dimensional JB$^*$-triple.
In the above results,
$\| f\|_0=1$ and $\det Df(0)=1$ are assumed for the Bloch mappings $f$,
where $Df$ is the Fr\'{e}chet derivative of $f$ and
$\| f\|_0$ is the prenorm of $f$ given by
\[
\| f\|_0=\sup \left\{
(1- \|z\|^2)^{c(\B)/n}|\det Df(z)|^{1/n}: z\in \B
\right\}.
\]
Here, $c(\B)$ is a constant which was defined in
\cite{HHK12}.
Since the prenorm and $\det Df(0)$ cannot be used in infinite dimensional case,
we give a distortion theorem and a lower bound for the Bloch constant
for Bloch mappings on $\B$ of the form $F(x)=f(x)x$ with $DF(0)=I$.
For the proof of a lower bound for the Bloch constant, we use a criterion for univalence that will be obtained in this paper.

Let $f$ be a normalized univalent function on $\U$
and let $g(\zeta)=\zeta f'(\zeta)$, $\zeta \in \U$.
Alexander \cite{Al15} gave a connection between convex functions and starlike functions on $\U$
as follows:
\[
f \mbox{ is convex} \Longleftrightarrow g \mbox{ is starlike}.
\]
Let $F_a$ be the mapping given by (\ref{example-1})
and let $G_a(z)=DF_a(z)z$.
Then
\[
G_a(z_1,z_2)=(z_1+2az_2^2, z_2),
\quad
(z_1,z_2)\in \C^2.
\]
So, on the Euclidean unit ball $\B^2$,
$F_a$ is convex if and only if $|a|\leq 1/2$ (see \cite[Example 9]{Su77}),
and
$G_a$ is starlike if and only if $|a|\leq 3\sqrt{3}/4$ (see \cite[Example 3]{Su77} and \cite{GK03}).
If $a\in \mathbb C$ such that $|a| \in (1/2,  3\sqrt{3}/4)$, then $F_a$ and $G_a$ as above
give an counterexample for the following equivalence:
\[
F_a \mbox{ is convex} \Longleftrightarrow G_a \mbox{ is starlike}.
\]
Thus, Alexander's theorem does not hold in general in higher dimensions.
Moreover,
let $H_a$ be a holomorphic mapping on $\C^2$
defined by
\[
H_a(z_1, z_2)=(z_1+a z_1z_2, z_2), \quad
(z_1,z_2)\in \C^2.
\]
Then $H_a$ is starlike on $\B^2$ if and only if $|a|\leq 1$ (see \cite[Example 6]{RS99} and \cite{GK03})
and quasi-convex of type $B$ on $\B^2$ if and only if $|a|\leq \sqrt{2}/2$ (see \cite[Example 12]{RS99}).
Therefore,  Alexander's type theorem
\[
H_a \mbox{ is quasi-convex of type } B \Longleftrightarrow DH_a(z)z \mbox{ is starlike},
\]
does not hold for some $a\in \mathbb C$.
We will show that Alexander's type theorem holds for
locally biholomorphic mappings $F$ of the form $F(x)=f(x)x$
and $G(x)=DF(x)x$
using the criteria for starlikeness and quasi-convexity that will be obtained in this paper.

\section{Preliminaries}
\label{prelim}

Let $X$ and $Y$ be complex Banach spaces. Let $X^k$ denote the product of $X$ with itself of $k$-times.
Throughout this paper,
let $B_r=\{ x\in X: \| x\|<r\}$
and $\U_r=\{ \zeta \in \C: |\zeta|<r\}$
for $r\in (0,1)$.
The unit ball $B_1$ will be denoted by $\B$
and the unit disc $\U_1$ will be denoted by $\U$.

\begin{definition}
\label{homogeneous polynomial}
Let
$k\in \N$.
A mapping $P:X\to Y$ is called a homogeneous polynomial %mapping
 of degree $k$
if
there exists a $k$-linear
mapping $u$ from $X^k$ into $Y$ such that
$$P(x)=u(x,\ldots,x) $$ for every $x \in X$.
\end{definition}

For two given domains $\mathcal{D}_1\subset X$, $\mathcal{D}_2\subset Y$, let $H(\mathcal{D}_1,\mathcal{D}_2)$ be the set of holomorphic
mappings from $\mathcal{D}_1$ into $\mathcal{D}_2$. For $F\in H(\mathcal{D}_1,\mathcal{D}_2)$ and $x\in \mathcal{D}_1$,
let $D^kF(x)$
denote the $k$-th Fr\'{e}chet derivative of $F$ at $x$.
If $\mathcal{D}_1$ contains the origin, then any mapping $F\in H(\mathcal{D}_1,\mathcal{D}_2)$
can be expanded into the series
\begin{equation}
\label{expansion}
F(x) =\sum_{k=0}^{\infty}\frac{1}{k!}D^k F(0)(x^k)
\end{equation}
in a neighbourhood of the origin.
Since $\frac{1}{k!}D^k F(0)(z^k)$ is a homogeneous polynomial %mapping
 of degree $k$,
we use the notation $P_k(x)=\frac{1}{k!}D^k F(0)(x^k)$ in this paper.
Let
\[
\| P_k\|=\sup_{\| x\|=1}\| P_k(x)\|.
\]

Let $L(X,\C)$ denote the set of continuous linear operators from $X$
into $\C$.
For each $x\in X\setminus\{0\}$, let
\[
T(x)=\{l_x\in L(X,\mathbb{C}):\ l_x(x)=\| x\|,\ \|l_x\|=1\}.
\]
This set is nonempty by the Hahn-Banach theorem.

A holomorphic mapping $F:\B \to X$ is said to be normalized
if $F(0)=0$ and $DF(0)=I$, where
$I$ is the identity operator on $X$.
A holomorphic mapping $F:B_r \to X$ is said to be starlike if
$F$ is biholomorphic on $B_r$ and $F(B_r)$ is a starlike domain with respect to the origin.
A locally biholomorphic mapping $F$ on $B_r$ is starlike if and only if
\begin{equation}
\label{starlike}
\Re l_x([DF(x)]^{-1}F(x))>0,
\quad x\in B_r\setminus \{ 0\},
\end{equation}
for $l_x\in T(x)$ (see \cite{GK03, Gu75, HK02, Su77}).
%A holomorphic mapping $F:B_r \to X$ is said to be convex if
%$F$ is biholomoprhic on $B_r$ and $F(B_r)$ is a convex domain.
A normalized locally biholomorphic mapping $F:\B\to X$ is said to be
quasi-convex of type $B$
if
\[
\Re l_x\left([DF(x)]^{-1}(D^2F(x)(x^2)+DF(x)x)\right)>0,
\quad
\,x\in \B\setminus \{ 0\},
\]
for $l_x\in T(x)$ (see \cite{H2023RM}, cf. \cite{GK03, RS99} for $X=\mathbb{C}^n$).

For $F\in H(\B,X)$ and $k\in \N$,
we say that $z=0$ is a {\it zero of order $k$ of $F$} if
$F(0)=0$, $DF(0)=0$, $\dots$, $D^{k-1}F(0)=0$,
but $D^kF(0)\neq 0$.

%Let $B_{X}$ and $B_{Y}$ be the unit balls of the complex
%Banach spaces $X$ and $Y$, respectively.
A holomorphic mapping $V:\B\to \B$ with $V(0)=0$
is called a {\it Schwarz mapping}.
We note that if $V$ is a Schwarz mapping such that $z=0$ is a zero of order $k$ of $V$,
then the following estimation holds (see e.g. \cite[Lemma 6.1.28]{GK03}):
\begin{equation}
\label{Schwarz-k}
\| V(x)\|\leq \| x\|^k,
\quad x\in \B.
\end{equation}

Let $F:\B\to X$ and $G:\B\to X$ be two holomorphic mappings.
We say that $G$ is subordinate to $F$ if there exists a Schwarz mapping $V$ on $\B$
such that $G=F\circ V$.
Consequently,
when $G$ is subordinate to $F$, we have
$\| DG(0)\|\leq \| DF(0)\|$.
If $F$ is biholomorphic on $\B$,
then $G$ is subordinate to $F$ if and only if $G(0)=F(0)$ and $G(\B)\subset F(\B)$.
Let $S(F)$ denote the class of all mappings $G:\B\to X$ which are subordinate to $F$.

Let $F\in H(\B,X)$ be defined by
\begin{equation}
\label{F}
F(x)=f(x)x,
\quad f\in H(\B,\C).
\end{equation}

Differentiating (\ref{F}), we have
\begin{equation}
\label{formula-1}
DF(x)=f(x)I+{xDf(x)}.
\end{equation}
If $f(x)\neq 0$ and $f(x)+Df(x)x\neq 0$, then $[DF(x)]^{-1}$ exists and
\begin{equation}
\label{formula-2}
[DF(x)]^{-1}=\frac{1}{f(x)}\left( I-\frac{xDf(x)}{f(x)+Df(x)x}\right)
\end{equation}
and thus
\begin{equation}
\label{formula-3}
[DF(x)]^{-1}F(x)=\frac{f(x)}{f(x)+Df(x)x}x.
\end{equation}

Let $u\in \partial \B$ be fixed and let
$f_{u}(\zeta)=\zeta f(\zeta u)$, $\zeta \in \U$.
Then $f_u$ is a holomorphic function on $\U$ and
\begin{equation}
\label{formula-4}
f_u'(\zeta)=f(\zeta u)+\zeta Df(\zeta u)u,
\end{equation}
which implies that
\begin{equation}
\label{formula-5}
DF(\zeta u)\zeta u=f_u'(\zeta)\zeta u
\end{equation}
and
\begin{equation}
\label{formula-6}
[DF(\zeta u)]^{-1}F(\zeta u)=\frac{f_u(\zeta)}{f_u'(\zeta)} u.
\end{equation}

Differentiating (\ref{formula-1}), we have
\[
D^2F(x)(x^2)=2(Df(x)x)x+{xD^{2}f(x)(x^2)}.
\]
Therefore, we have
\[
[DF(x)]^{-1}D^2F(x)(x^2)=\frac{2Df(x)x+D^{2}f(x)(x^2)}{f(x)+Df(x)x}x
\]
and
\begin{eqnarray}
\label{formula-7}
\lefteqn{[DF(x)]^{-1}(D^2F(x)(x^2)+DF(x)x)}
\\ \nonumber
&&\quad
=\left(\frac{2Df(x)x+D^{2}f(x)(x^2)}{f(x)+Df(x)x}+1\right)x.
\end{eqnarray}
Also, differentiating (\ref{formula-4}), we have
\begin{equation}
\label{formula-8}
f_u''(\zeta)=2Df(\zeta u)u+\zeta D^2f(\zeta u)(u^2).
\end{equation}

\section{Main results}
\label{main}

\subsection{Univalence, starlikeness and quasi-convexity criteria}
\label{univalence}

Let $F\in H(\mathbb B,X)$, $F(x)=f(x)x$, $x\in \mathbb B$,
where $f\in H(\B,\C)$.
For each $u\in \partial \B$, let again $f_{u}(\zeta)=\zeta f(\zeta u)$.

Liu \cite[Lemma 2.1]{L22}
showed that if
$F(x)=f(x)x$ is a normalized biholomorphic mapping on $\B$,
then
$f_u$ is a normalized univalent function  on $\U$ for all $u\in \partial \B$.
Moreover, we obtain the following theorem.

\begin{theorem}
\label{thm-univalence}
Let $F\in H(\mathbb B,X)$, $F(x)=f(x)x$, $x\in \mathbb B$,
where  $f\in H(\B,\C)$.
Let $r\in (0,1)$ be fixed.
Then $F$ is a biholomorphic mapping on $B_r$ if and only if
$f_u$ is univalent on $\U_r$ for all $u\in \partial \B$.
\end{theorem}

\begin{proof}
It is clear that $F$ is univalent on $B_r$ if and only if $f_u$ is univalent on $\U_r$ for all $u\in \partial \B$.
Assume that $f_u$ is univalent on $\U_r$ for all $u\in \partial \B$.
Combining the univalence of $f_u$ and  (\ref{formula-4}),
we obtain that
$f(x)\neq 0$ and $f(x)+Df(x)x\neq 0$ for all $x\in B_r$.
Therefore, from (\ref{formula-2}), $[DF(x)]^{-1}$ exists for all $x\in B_r$.
Since $F$ is univalent and locally biholomorphic on $B_r$,
$F$ is biholomorphic on $B_r$.
This completes the proof.
\end{proof}

Next, we give a starlikeness criterion for holomorphic mappings $F$ on $\B$
of the form $F(x)=f(x)x$.

\begin{theorem}
\label{thm-starlike}
Let $F\in H(\mathbb B,X)$, $F(x)=f(x)x$, $x\in \mathbb B$,
where  $f\in H(\B,\C)$.
Let $r\in (0,1)$ be fixed.
Then $F$ is a starlike mapping on $B_r$ if and only if
$f_u$ is starlike on $\U_r$ for all $u\in \partial \B$.
\end{theorem}

\begin{proof}
From Theorem \ref{thm-univalence},
we obtain that $F$ is biholomorphic on $B_r$ if and only if $f_u$ is univalent on $\U_r$ for all $u\in \partial \B$.
Further, from (\ref{formula-6}), we obtain that
\[
\Re l_{\zeta u}\left([DF(\zeta u)]^{-1}F(\zeta u)\right)=\Re \frac{f_u(\zeta)}{\zeta f_u'(\zeta)},
\quad \zeta \in \U\setminus \{ 0\}
\]
for all $u\in \partial \B$.
Then by (\ref{starlike}),
$F$ is starlike on $B_r$ if and only if $f_u$ is starlike on $\U_r$ for all $u\in \partial \B$.
This completes the proof.
\end{proof}

Next, we give a quasi-convexity criterion for holomorphic mappings $F$ on $\B$
of the form $F(x)=f(x)x$ with $f(0)=1$.

\begin{theorem}
\label{thm-convex}
Let $F\in H(\mathbb B,X)$, $F(x)=f(x)x$, $x\in \mathbb B$,
where  $f\in H(\B,\C)$ with $f(0)=1$.
Let $r\in (0,1)$ be fixed.
Then $F$ is a quasi-convex mapping of type $B$ on $B_r$ if and only if
$f_u$ is convex on $\U_r$ for all $u\in \partial \B$.
\end{theorem}

\begin{proof}
By (\ref{formula-4}), (\ref{formula-7}) and (\ref{formula-8}),
we have
\begin{eqnarray*}
%\label{formula-9}
\lefteqn{l_{\zeta u}\left([DF(\zeta u)]^{-1}(D^2F(\zeta u)(\zeta^2 u^2)+DF(\zeta u)\zeta u)\right)}
\\
&&\quad
=l_{\zeta u}\left(\left(\frac{2Df(\zeta u)\zeta u+D^{2}f(\zeta u)(\zeta^2 u^2)}{f(\zeta u)+Df(\zeta u)(\zeta u)}+1\right)\zeta u\right)
\\
&& \quad
=\left(\frac{\zeta f_u''(\zeta)}{f_u'(\zeta)}+1\right)|\zeta|,
\quad
\zeta \in \U\setminus \{ 0\}.
\end{eqnarray*}
Thus, we conclude that
$F$ is a quasi-convex mapping of type $B$ on $B_r$ if and only if
$f_u$ is convex on $\U_r$ for all $u\in \partial \B$.
This completes the proof.
\end{proof}

\subsection{Generalized Bieberbach conjecture}
\label{Bieberbach}

As a corollary of univalence criterion (Theorem \ref{thm-univalence}),
we can prove the generalized Bieberbach conjecture
(Rogosinski conjecture)
for mappings $F$ and $G$ on $\B$
of the form $F(x)=f(x)x$ and $G(x)=g(x)x$.

\begin{theorem}
\label{thm-Bieberbach}
Let $f,g \in H(\B,\C)$ and
let $F,G\in H(\mathbb B,X)$, $F(x)=f(x)x$, $G(x)=g(x)x$, $x\in \mathbb B$.
Let
\begin{equation}
\label{homogeneous expansion-g}
G(x)=\sum_{s=1}^{\infty}P_s(x),
\quad \mbox{near the origin},
\end{equation}
where $P_s$ is a homogeneous polynomial of degree $s$ on $X$.
Assume that $F$ is biholomorphic on $\B$ and $G\in S(F)$.
Then we have
\begin{equation}
\label{eq-Bieberbach}
\| P_s(w)\|\leq s\| DF(0)\|,
\quad \forall \,
s\geq 1, \| w\|=1.
\end{equation}
The right-hand side of $(\ref{homogeneous expansion-g})$
converges uniformly on $r\B$ for each $r\in (0,1)$ and
$(\ref{homogeneous expansion-g})$ holds for all $x\in \B$.
Moreover, if $G=F$,
then the estimation $(\ref{eq-Bieberbach})$
is sharp.
\end{theorem}

\begin{proof}
Since $G(x)=g(x)x$,
we have $P_1(x)=g(0)x$ and
$P_k(x)=Q_{k-1}(x)x$ for $k\geq 2$,
where
\[
g(x)=g(0)+\sum_{k=1}^{\infty}Q_k(x)
\]
and $Q_k$ is a $\C$-valued homogeneous polynomial of degree $k$ on $X$.
Let $u\in \partial \B$ be arbitrarily fixed.
Let $f_{u}(\zeta)=\zeta f(\zeta u)$ and $g_{u}(\zeta)=\zeta g(\zeta u)$.
Then from Theorem \ref{thm-univalence},
$f_u$ is a univalent function on $\U$.
Since $g_u\in S(f_u)$ and
\[
g_u(\zeta)=g(0)\zeta+\sum_{k=1}^{\infty}Q_k(u)\zeta^{k+1},
\]
by the generalized Bieberbach conjecture on the unit disc $\U$,
we have $|g(0)|\leq |f(0)|$ and
\[
|Q_{k-1}(u)|\leq k|f(0)|
\]
for all $k\geq 2$,
which implies that
\[
\| P_k(u)\|\leq k|f(0)|=k\| DF(0)\|
\]
for all $k\geq 1$.
Since $u\in \partial \B$ is arbitrary,
we obtain (\ref{eq-Bieberbach}),
as desired.

Next, from (\ref{eq-Bieberbach}),
we have
\[
\left\| P_s(x)\right\|\leq s\| DF(0)\|\| x\|^s,
\quad  \forall \,
s\geq 1, x\in \B,
\]
which implies that the right hand side of (\ref{homogeneous expansion-g})
converges uniformly on $r\B$ and by the identity theorem for holomorphic mappings (see e.g. \cite[Theorem 2.7]{ReSh05}),
(\ref{homogeneous expansion-g}) holds on $r\B$ for each $r\in (0,1)$.

Finally, we prove the sharpness.
Let $v\in \partial \B$, $l_v\in T(v)$ be fixed and
let
\begin{equation}
\label{Koebe}
F(x)=\frac{1}{(1-l_v(x))^2}x,
\quad
x\in \B.
\end{equation}
Since
\[
f_u(\zeta)=\frac{\zeta}{(1-l_v(u)\zeta)^2}
\]
is univalent on $\U$ for all $u\in \partial \B$,
$F$ is biholomorphic on $\B$
by Theorem \ref{thm-univalence}.
Clearly, $DF(0)=I$.
Since $P_k(v)=kv$,
the equality in (\ref{eq-Bieberbach})
is attained by this mapping.
This completes the proof.
\end{proof}

In particular, when $F$ is a normalized biholomorphic mapping of the form $F(x)=f(x)x$ and $G=F$,
we obtain the Bieberbach conjecture
as a corollary of univalence criterion (Theorem \ref{thm-univalence}).

\begin{corollary}
\label{cor-Bieberbach}
Let $f\in H(\B,\C)$ and
let $F\in H(\mathbb B,X)$, $F(x)=f(x)x$, $x\in \mathbb B$.
Let
\begin{equation}
\label{homogeneous expansion-f}
F(x)=x+\sum_{s=2}^{\infty}P_s(x),
\quad \mbox{near the origin},
\end{equation}
where $P_s$ is a homogeneous polynomial of degree $s$ on $X$.
Assume that $F$ is biholomorphic on $\B$.
Then we have
\begin{equation}
\label{eq-Bieberbach-cor}
\| P_s(w)\|\leq s,
\quad \forall \,
s\geq 1, \| w\|=1
\end{equation}
The right-hand side of $(\ref{homogeneous expansion-f})$
converges uniformly on $r\B$ for each $r\in (0,1)$ and
$(\ref{homogeneous expansion-f})$ holds for all $x\in \B$.
Moreover,
the estimation $(\ref{eq-Bieberbach-cor})$
is sharp.
\end{corollary}

\subsection{Covering theorem and distortion theorem}
\label{covering and distortion}

Liu \cite[Theorem 2.2]{L22}
obtained the following growth theorem
for normalized biholomorphic mappings $F$ on $\B$ of the form
$F(x)=f(x)x$,
where $f\in H(\B,\C)$.

\begin{lemma} \label{growth-lemma}
Let $F\in H(\mathbb B,X)$ be a biholomorphic mapping given by $F(x)=f(x)x$, $x\in \mathbb B$,
where $f\in H(\B,\C)$ with $f(0)=1$.
Then
\[
\frac{\|x\|}{(1+\|x\|)^2}
\leq \|F(x)\|\leq
\frac{\|x\|}{(1-\|x\|)^2},\quad x\in {\B}.
\]
These bounds are sharp.
\end{lemma}

We
obtain the following covering theorem
for normalized biholomorphic mappings
$F$ of the form $F(x)=f(x)x$
on the unit ball $\B$ in a complex
Banach space
as a corollary of univalence criterion (Theorem \ref{thm-univalence}).

\begin{theorem} \label{covering-lemma}
Let $F\in H(\mathbb B,X)$ be a biholomorphic mapping given by $F(x)=f(x)x$, $x\in \mathbb B$,
where $f\in H(\B,\C)$ with $f(0)=1$.
Then $F({\B})$ contains the ball ${B}_{1/4}$.
The constant $1/4$ is the best possible.
\end{theorem}

\begin{proof}
Let $u\in \partial \B$ be arbitrarily fixed
and
let $f_{u}(\zeta)=\zeta f(\zeta u)$.
Then $f_u$ is a normalized univalent function on the unit disc $\U$
by Theorem \ref{thm-univalence}
and therefore, $f_u(\U)\supset \U_{1/4}$.
Since $F(\zeta u)=f_u(\zeta)u$,
$F(\B)\supset B_{1/4}$.
By considering $F(\zeta v)$,
where $F$ is the biholomorphic mapping on $\B$ defined in (\ref{Koebe}),
we obtain that
the constant $1/4$ is the best possible.
This completes the proof.
\end{proof}

Liu \cite[Theorem 2.2]{L22} also
obtained the following sharp distortion theorem for normalized biholomorphic mappings $F$ on $\B$
of the form (\ref{F}):
\begin{equation}
\label{distortion-Liu}
\frac{\| x\|(1-\| x\|)}{(1+\| x\|)^3}
\leq
\| DF(x) x\|
\leq
\frac{\| x\|(1+\| x\|)}{(1-\| x\|)^3}.
\end{equation}
If $\B=\B_{\mathcal H}$ is the unit ball of a complex Hilbert space ${\mathcal H}$,
we further obtain the following sharp distortion theorem
for normalized biholomorphic mappings $F$ of the form (\ref{F})
by applying the Bieberbach conjecture (Corollary \ref{cor-Bieberbach}). Recall that
\[
\| DF(x)\|=\sup_{\| \xi\|=1}\| DF(x)\xi\|.
\]

\begin{theorem}
\label{thm-distortion}
Let ${\mathcal H}$ be a complex Hilbert space and
let $F\in H(\B_{\mathcal H},{\mathcal H})$ be a biholomorphic mapping
given by
$F(x)=f(x)x$, $x\in \B_{\mathcal H}$,
where $f\in H(\B_{\mathcal H},\C)$ with $f(0)=1$.
Then
\begin{equation}
\label{eq-distortion}
\frac{1-\| x\|}{(1+\| x\|)^3}
\leq
\| DF(x)\|
\leq
\frac{1+\| x\|}{(1-\| x\|)^3},
\quad
x\in \B_{\mathcal H}.
\end{equation}
These estimations are sharp.
\end{theorem}

\begin{proof}
Since $\| DF(x)x\|/\| x\|\leq \| DF(x)\|$ holds for all $x\in \B_{\mathcal H}\setminus \{ 0\}$,
the lower bound in (\ref{eq-distortion}) follows from (\ref{distortion-Liu}).

Next, we will show the upper bound in (\ref{eq-distortion}). To this aim, we use the property that $F$
can be expanded into the series
\begin{equation}
\label{expansion-3}
F(x) =\sum_{k=0}^{\infty}\frac{1}{k!}D^k F(0)(x^k),
\quad
x\in \B_{\mathcal H}
\end{equation}
by Corollary \ref{cor-Bieberbach}.
Note that
\begin{eqnarray}
\label{eq-homogeneous}
\sup_{\| x\|=\| \xi\|=1}\left\|D^k F(0)(x^{k-1},\xi)\right\|
&=&\sup_{\| x\|=1}\left\|D^k F(0)(x^{k})\right\|
\\ \nonumber
&=&
\| D^k F(0)\|
\end{eqnarray}
holds on complex Hilbert spaces
by \cite[Lemma 1 and Theorem 4]{H54}.
Therefore, from
(\ref{eq-Bieberbach-cor}),
(\ref{expansion-3})
and
(\ref{eq-homogeneous}),
we have
\begin{eqnarray*}
\| DF(x)\xi\|
&=&
\left\| \xi+\sum_{k=2}^{\infty}\frac{k}{k!}D^k F(0)(x^{k-1},\xi)\right\|
\\
&\leq &
\| \xi \|+\sum_{k=2}^{\infty}\frac{k}{k!}\| D^k F(0)\| \|x\|^{k-1}\|\xi\|
\\
&\leq &
1+\sum_{k=2}^{\infty}k^2 \| x\|^{k-1}
\\
&=&
\frac{1+\| x\|}{(1-\| x\|)^3}
\end{eqnarray*}
for $x\in \B_{\mathcal H}$ and $\xi \in {\mathcal H}$ with $\| \xi\|=1$.

Finally, we prove the sharpness.
Let $v\in \partial \B_{\mathcal H}$ be fixed and
let
\[
F(x)=\frac{1}{(1-\langle x, v\rangle)^2}x,
\quad
x\in \B_{\mathcal H},
\]
where $\langle \cdot, \cdot \rangle$
is the inner product on ${\mathcal H}$.
$F$ is a biholomorphic mapping on $\B_{\mathcal H}$ with $DF(0)=I$.
Since
\[
DF(r v)v
=
\left( 1+\sum_{k=2}^{\infty}k^2 r^{k-1}\right)v
=\frac{1+r}{(1-r)^3}v
\]
holds for $r\in (-1,1)$,
the equalities in the upper and lower bounds of (\ref{eq-distortion})
hold by this mapping.
This completes the proof.
\end{proof}

\subsection{Fekete-Szeg\"{o} inequality}
\label{Fekete-Szego}

As a corollary of univalence criterion (Theorem \ref{thm-univalence}),
we obtain the Fekete-Szeg\"{o} inequality
for biholomorphic mappings $F$ on $\B$
of the form $F(x)=f(x)x$.

\begin{theorem}
\label{thm-Fekete}
Let $F\in H(\mathbb B,X)$ be a biholomorphic mapping given by $F(x)=f(x)x$, $x\in \mathbb B$,
where $f\in H(\B,\C)$ with $Df(0)=1$.
Let
\[
F(x)=x+\sum_{k=2}^{\infty}\frac{1}{k!}D^k F(0)(x^k)=x+\sum_{k=2}^{\infty}P_k(x),
\quad
x\in \B,
\]
be the homogeneous polynomial expansion of $F$.
Then for any $u \in \partial \B$ and $\lambda \in [0,1)$, we have
\begin{eqnarray*}
%\label{new-Fekete-3}
\lefteqn{\left\| \frac{1}{3!}D^{3} F(0)(u^3)-\lambda \frac{1}{2!}D^{2} F(0)\left(u, \frac{1}{2!}D^{2} F(0)(u^2)\right)\right\|}
\\ %\nonumber
&& \qquad
\leq
1+2\exp\left(-\frac{2\lambda}{1-\lambda}\right).
\end{eqnarray*}
This estimation is sharp.
\end{theorem}

\begin{proof}
Let
\[
f(x)=1+\sum_{k=1}^{\infty}Q_k(x),
\]
where $Q_k$ is a $\C$-valued homogeneous polynomial of degree $k$ on $X$,
and let $f_{u}(\zeta)=\zeta f(\zeta u)$.
Then as in the proof of Theorem \ref{thm-Bieberbach},
we have $P_k(x)=Q_{k-1}(x)x$ for $k\geq 2$
and $f_u$ is a normalized univalent function on the unit disc $\U$ with
\[
f_u(\zeta)=\zeta+\sum_{k=1}^{\infty}Q_k(u)\zeta^{k+1}.
\]
By the Fekete-Szeg\"{o} inequality
for normalized univalent functions on the unit disc $\U$, we have
\begin{equation}
\label{Fekete-disc}
|Q_2(u)-\lambda Q_1(u)^2|
\leq
1+2\exp\left(-\frac{2\lambda}{1-\lambda}\right)
\end{equation}
for all $\lambda\in [0,1)$.
Therefore, we have
\begin{eqnarray*}
\lefteqn{\left\| \frac{1}{3!}D^{3} F(0)(u^3)-\lambda \frac{1}{2!}D^{2} F(0)\left(u, \frac{1}{2!}D^{2} F(0)(u^2)\right)\right\|}
\\ %\nonumber
&& \qquad
=\left\| Q_2(u)u-\lambda \frac{1}{2!}D^{2} F(0)(u, Q_1(u)u)\right\|
\\
&& \qquad
=\left\| Q_2(u)u-\lambda Q_1(u)^2 u\right\|
\\
&& \qquad
=
|Q_2(u)-\lambda Q_1(u)^2|
\\
&& \qquad
\leq
1+2\exp\left(-\frac{2\lambda}{1-\lambda}\right)
\end{eqnarray*}
for all $\lambda\in [0,1)$.

Finally, we give a proof for sharpness.
Let $\lambda \in [0,1)$ be fixed and
let $f_{\lambda}$ be a normalized univalent function on the unit disc $\U$ which gives equality in the Fekete-Szeg\"{o} inequality
for $\lambda$.
For fixed $v\in \partial \B$, $l_v\in T(v)$, let
\[
F_{\lambda}(x)=\frac{f_{\lambda}(l_v(x))}{l_v(x)}x,
\quad
x\in \B.
\]
Since for any $u\in \partial \B$,
\[
\zeta\frac{f_{\lambda}(l_v(\zeta u))}{l_v(\zeta u)}
=
\left\{
\begin{array}{ll}
\frac{1}{l_v(u)}f_{\lambda}(\zeta l_v(u)), & l_v(u)\neq 0; \\
\zeta, & l_v(u)=0
\end{array}
\right.
\]
is a univalent function on $\U$,
$F_{\lambda}$ is biholomorphic on $\B$ by Theorem \ref{thm-univalence}.
Also, we have
$F_{\lambda}(\zeta v)=f_{\lambda}(\zeta)v$
and thus
\begin{eqnarray*}
\lefteqn{\left\| \frac{1}{3!}D^{3} F_{\lambda}(0)(v^3)-\lambda \frac{1}{2!}D^{2} F_{\lambda}(0)\left(u, \frac{1}{2!}D^{2} F_{\lambda}(0)(v^2)\right)\right\|}
\\
&& \qquad
=
|Q_2(v)-\lambda Q_1(v)^2|
\\
&& \qquad
=
1+2\exp\left(-\frac{2\lambda}{1-\lambda}\right).
\end{eqnarray*}
This completes the proof.
\end{proof}

\subsection{Bohr-Rogosinski inequality}
\label{Bohr-Rogosinski inequality}

For holomorphic mappings $G:\B\to \overline{\B}$,
the Bohr-Rogosinski inequality can be written as
\[
\sum_{s=N}^{\infty}\| P_s(x)\|\leq 1-\| G(x)\|={\rm dist}(G(x), \partial \B).
\]
As in \cite{KKP21}, for a biholomorphic mapping $F:\B\to X$,
we say that the family $S(F)$ has a Bohr-Rogosinski phenomenon
if there exists an $r_N^F$, $0<r_N^F\leq 1$
such that for $G(x)=\sum_{s=0}^{\infty}P_s(x)\in S(F)$,
we have
\[
\| G(x)\|+\sum_{s=N}^{\infty}{\| P_s(x)\|}
\leq
\|F(0)\|+{\rm dist}(F(0),\partial F(\B))
\]
for $\| x\|=r\leq r^F_{N}$.

As a corollary of the generalized Bieberbach conjecture (Theorem \ref{thm-Bieberbach})
and the covering theorem (Theorem \ref{covering-lemma}),
we obtain the following theorem.

\begin{theorem}
\label{thm1-subord}
Let $\B$ be the unit ball of a complex Banach space $X$ and let $a\in X$ be given.
Let $f,g \in H(\B,\C)$ and
let $F,G\in H(\B, X)$ be given by
$F(x)=a+f(x)x$, $G(x)=a+g(x)x$, $x\in \B$.
Assume that $F$ is biholomorphic on $\B$ and $G\in S(F)$.
Let
\[
G(x)=a+\sum_{s=1}^{\infty}P_s(x),
\quad x\in \B,
\]
where $P_s$ is a homogeneous polynomial of degree $s$ on $X$.
Then for $m, N\in\mathbb{N}$, we have
\[
\| G(V(x))\|+\sum_{s=N}^{\infty}{\| P_s(x)\|}
\leq
|F(0)|+{\rm dist}(F(0),\partial F(\B))
\]
for $\| x\|=r\leq r^F_{m, N}$,
where $V:\B\to \B$ is a  Schwarz mapping having $z=0$ a zero of order $m$
and $  r^F_{m,N}$ is the unique root in $(0,1)$ of the equation
\[
4r^m-(1-r^m)^2+4r^N[N(1-r)+r]\left(\frac{1-r^m}{1-r}\right)^2=0.
\]
%When $B$ is the unit polydisc,
The number $ r^F_{m, N}$ cannot be improved.
\end{theorem}

\begin{proof}
Let $H=[DF(0)]^{-1}(F-F(0))$.
Then $H$ is a normalized biholomorphic mapping on $\B$
with
\[
H(x)=\frac{f(x)}{f(0)}x.
\]
By Theorem \ref{covering-lemma},
$H({\B})$ contains the ball ${B}_{1/4}$.
This implies that
\[
\frac{f(0)}{4}\B=\frac{1}{4}DF(0)(\B)\subset \Omega-F(0),
\]
where $\Omega=F(\B)$.
Then, we have
\begin{equation}
\label{eq-DF}
\frac{1}{4}\| DF(0)\| \leq {\rm dist}(F(0), \partial \Omega),
\end{equation}
which combined with Theorem \ref{thm-Bieberbach} implies that
\[
\| P_s(u)\|\leq 4s\cdot {\rm dist}(F(0), \partial \Omega),
\quad  \forall \,
s\geq 1, \| u\|=1.
\]
Therefore, we have
\begin{eqnarray}
\label{eq-subord-1}
\sum_{s=N}^{\infty}{\| P_s(x)\|}
&\leq&
4{\rm dist}(F(0), \partial \Omega)\sum_{s=N}^{\infty}sr^s
\\ \nonumber
&=&
{\rm dist}(F(0), \partial \Omega)\frac{4r^N[N(1-r)+r]}{(1-r)^2}
\end{eqnarray}
for $\| x\|=r$.
Moreover, because $G\prec F$, it follows from Lemma \ref{growth-lemma}
and (\ref{eq-DF}) that
\[
\| G(x)-G(0)\|\leq \| DF(0)\|\frac{r}{(1-r)^2}
\leq
{\rm dist}(F(0), \partial \Omega)\frac{4r}{(1-r)^2},
\]
for $\| x\|=r$.
Therefore, from $G(0)=F(0)$ and (\ref{Schwarz-k}), for $m\in \N$, we have
\begin{equation}
\label{eq-subord-2}
\| G(V(x))\|\leq
\| F(0)\|+{\rm dist}(F(0), \partial \Omega)\frac{4r^m}{(1-r^m)^2}.
\end{equation}
By (\ref{eq-subord-1}) and (\ref{eq-subord-2}), we deduce that
\begin{eqnarray*}
\lefteqn{\| G(V(x))\|+\sum_{s=N}^{\infty}{\| P_s(x)\|}}
\\
&&\leq
\| F(0)\|+{\rm dist}(F(0), \partial \Omega)\left[\frac{4r^m}{(1-r^m)^2}+
\frac{4r^N[N(1-r)+r]}{(1-r)^2}\right]
\\
&&\leq
\| F(0)\|+{\rm dist}(F(0), \partial \Omega)
\end{eqnarray*}
provided
\[
\frac{4r^m}{(1-r^m)^2}+
\frac{4r^N[N(1-r)+r]}{(1-r)^2}\leq 1,
\]
which is equivalent to $r\leq  r^F_{m,N}$.

Finally, we will show that the number $r^F_{m,N}$ is sharp.
For fixed $v\in \partial \B$ and $l_v\in T(v)$,
let $F$ be the normalized biholomorphic mapping on $\B$ defined in (\ref{Koebe}).
Also, let $G(x)=F(x)$ and $V(x)=l_v(x)^{m-1}x$.
Then we have
\[
{\| G(V(x))\|+\sum_{s=N}^{\infty}{\| P_s(x)\|}}
=\frac{r^m}{(1-r^m)^2}+\frac{r^N[N(1-r)+r]}{(1-r)^2}
\]
for $x=rv$
and
\[
\| F(0)\|+{\rm dist}(F(0), \partial \Omega)=\frac{1}{4}.
\]
Thus, the number $r^F_{m,N}$ is sharp.
This completes the proof.
\end{proof}

\begin{remark}
(i) Note that $r^F_{m,N}$ depends only on $m$ and $N$.

(ii)
In particular, if $X=\C$,
then Theorem \ref{thm1-subord}
reduces to \cite[Theorem 5]{KKP21}.
\end{remark}

Letting $N=1$ and $m\to \infty$ in Theorem \ref{thm1-subord},
we obtain the following corollary,
which is a generalization of \cite[Theorem 1]{A10}
to the unit ball of a complex Banach space.

\begin{corollary}
\label{corollary-subord-0}
Let $\B$ be the unit ball of a complex Banach space $X$ and let $a\in X$ be given.
Let $f,g \in H(\B,\C)$ and
let $F,G\in H(\mathbb B, X)$ be given by
$F(x)=a+f(x)x$, $G(x)=a+g(x)x$, $x\in {\mathbb B}$.
Assume that $F$ is biholomorphic on $\B$ and $G\in S(F)$.
Let
\[
G(x)=a+\sum_{s=1}^{\infty}P_s(x),
\quad x\in \B,
\]
where $P_s$ is a homogeneous polynomial of degree $s$ on $X$.
Then we have
\[
\sum_{s=1}^{\infty}{\| P_s(x)\|}
\leq
{\rm dist}(F(0),\partial F(\B))
\]
for $\| x\|=r\leq 3-2\sqrt{2}$.
The number $3-2\sqrt{2}$ cannot be improved.
\end{corollary}

Letting $V(x)=x$ in Theorem \ref{thm1-subord},
we obtain the following corollary.

\begin{corollary}
\label{corollary-subord}
Let $\B$ be the unit ball of a complex Banach space $X$ and let $a\in X$.
Let $f,g \in H(\B,\C)$ and
let $F,G\in H(\mathbb B, X)$ be given by $F(x)=a+f(x)x, G(x)=a+g(x)x$, $x\in {\mathbb B}$.
Assume that $F$ is biholomorphic on $\B$ and $G\in S(F)$.
Let
\[
G(x)=a+\sum_{s=1}^{\infty}P_s(x),
\quad x\in\B,
\]
where $P_s$ is a homogeneous polynomial of degree $s$ on $X$.
Then for $m, N\in\mathbb{N}$, we have
\[
\| G(x)\|+\sum_{s=N}^{\infty}{\| P_s(x)\|}
\leq
\|F(0)\|+{\rm dist}(F(0),\partial F(\B))
\]
for $\| x\|=r\leq r^F_{N}$,
where $  r^F_{N}$ is the unique root in $(0,1)$ of the equation
\[
4r-(1-r)^2+4r^N[N(1-r)+r]=0.
\]
%When $B$ is the unit polydisc,
The number $ r^F_{N}$ cannot be improved.
\end{corollary}

\subsection{Bloch constant}
\label{Bloch constant}

For $F\in H(\B,X)$ and $a\in \B$, a {\it schlicht ball} of $F$ centered at $F(a)$ is a ball
in $X$ with
center $F(a)$ such that $F$ maps an open subset of $\B$ containing $a$ biholomorphically
onto this ball. For a point $a\in \B$, let $r(a; F)$ be the radius of the largest schlicht
ball of $F$ centered at $F(a)$.
Let $r(F)=\sup\{ r(F,a): a\in \B\}$.

Let
\[
\| F\|_{\mathcal{B}}=\| F(0)\|+\sup_{x\in \B\setminus \{ 0\}}(1-\|x\|^2)\frac{\|DF(x)x\|}{\| x\|}.
\]
We say that $F$ is a {\it Bloch mapping} on $\B$ if $\| F\|_{\mathcal{B}}<\infty$.

Let
\[
\mathcal{B}_1(X)=\{ F: F(x)=f(x)x, f\in H(\B,\C), f(0)=1, \| F\|_{\mathcal{B}}=1 \}
\]
and
\[
\mathbf{B}^*_{X}=\inf \{ r(F): F\in  \mathcal{B}_1(X) \}.
\]
Then
\[
\mathcal{B}_1(\C)=\{ f: f\in H(\U,\C), f(0)=0, f'(0)=1, \| f\|_{\mathcal{B}}=1\},
\]
where
\[
\| f\|_{\mathcal{B}}=| f(0)|+\sup_{\zeta\in \U}(1-|\zeta|^2)|f'(\zeta)|
\]
and $\mathbf{B}^{*}_{\C}=\mathbf{B}$ is the usual Bloch constant in one complex variable.

By applying univalence criterion (Theorem \ref{thm-univalence}),
we can generalize Bonk's distortion theorem \cite{B90}
and
the lower estimate for the Bloch constant in one complex variable
\cite{B90}
to Bloch mappings of the form $F(x)=f(x)x$ with $DF(0)=I$ on the unit ball
of infinite dimensional complex Banach spaces as follows.

\begin{theorem}
\label{theorem-Bloch}
\begin{enumerate}
\item[(i)]
Suppose $F\in \mathcal{B}_1(X)$. Then
\begin{equation}
\label{eq-Bonk-distortion}
\Re l_x(DF(x)x)\geq \frac{1-\sqrt{3}\| x\|}{\left(1-\frac{1}{\sqrt{3}}\| x\|\right)^3}\| x\|,
\quad \forall \, l_x\in T(x),\ 0<\| x\|\leq \frac{1}{\sqrt{3}};
\end{equation}
\item[(ii)]
$\mathbf{B}^*_{X}\geq \frac{\sqrt{3}}{4}$.
\end{enumerate}
\end{theorem}

\begin{proof}
(i) Assume that $F(x)=f(x)x\in \mathcal{B}_1(X)$.
Let $u\in \partial \B$ be arbitrarily fixed.
Then by (\ref{formula-5}), we have
$f_u\in \mathcal{B}_1(\C)$.
By Bonk's distortion theorem \cite{B90} on $\U$,
we have
\[
\Re f_u'(\zeta)\geq \frac{1-\sqrt{3}|\zeta|}{\left(1-\frac{1}{\sqrt{3}}|\zeta|\right)^3},
\quad
|\zeta|\leq \frac{1}{\sqrt{3}}.
\]
By using (\ref{formula-5}) again, we obtain
\begin{eqnarray*}
\Re l_{\zeta u}(DF(\zeta u)\zeta u)
&=&
\Re l_{\zeta u}(f_u'(\zeta)\zeta u)
\\
&=&
\Re f_u'(\zeta)| \zeta |
\\
&\geq &
\frac{1-\sqrt{3}|\zeta|}{\left(1-\frac{1}{\sqrt{3}}|\zeta|\right)^3}| \zeta |
\end{eqnarray*}
for $0<|\zeta|\leq 1/\sqrt{3}$.
Since $u \in \partial \B$ is arbitrary, we obtain
(\ref{eq-Bonk-distortion}), as desired.

(ii)
Let $F\in \mathcal{B}_1(X)$.
Then for any $u\in \partial \B$,
$f_u\in \mathcal{B}_1(\C)$.
By the proof of the lower estimate for the Bloch constant due to Bonk \cite{B90},
$f_u$ maps the disc $\U_{1/\sqrt{3}}$ conformally onto a simply connected domain $D_u$ such that $D_u$ contains the disc $\U_{\sqrt{3}/{4}}$.
By Theorem \ref{thm-univalence},
$F$ is a biholomorphic mapping on  $B_{1/\sqrt{3}}$
and $F(B_{1/\sqrt{3}})$ contains the ball $B_{{\sqrt{3}}/{4}}$.
This completes the proof.
\end{proof}

\subsection{Alexander's type theorem}
\label{Alexander's theorem}

As a corollary of starlikeness criterion (Theorem \ref{thm-starlike}) and quasi-convexity criterion (Theorem \ref{thm-convex}), we can prove Alexander's type theorem
for mappings $F$ of the form $F(x)=f(x)x$ with $f(0)=1$ and $G(x)=DF(x)x=(f(x)+Df(x)x)x$.

\begin{theorem}
\label{Alexander}
Let $F\in H(\mathbb B,X)$ be given by $F(x)=f(x)x$, $x\in \mathbb B$, 
where $f\in H(\B,\C)$ with $f(0)=1$,
and let $G(x)=DF(x)x$.
Then $F$ is quasi-convex of type $B$ on $\B$ if and only if $G$ is starlike on $\B$.
\end{theorem}

\begin{proof}
Let $g(x)=f(x)+Df(x)x$ and let
$g_u(\zeta)=\zeta g(\zeta u)$.
Then by (\ref{formula-4}), we have
\begin{eqnarray*}
g_u(\zeta)&=&\zeta g(\zeta u)
\\
&=&
\zeta(f(\zeta u)+Df(\zeta u)\zeta u)
\\
&=&
\zeta f_u'(\zeta).
\end{eqnarray*}
Therefore,
by Theorems \ref{thm-starlike}, \ref{thm-convex} and
Alexander's theorem \cite{Al15} on $\U$,
we conclude that
$F$ is quasi-convex of type $B$ on $\B$ if and only if $G$ is starlike on $\B$.
This completes the proof.
\end{proof}
%\section*{Acknowledgments}
%\label{acknowledgments}

%Hidetaka Hamada is partially supported by JSPS KAKENHI Grant
%Number JP19K03553.
%Tatsuhiro Honda is partially supported by
%JSPS KAKENHI Grant Number JP23K03136.

%\section*{Declaration of competing interest}
%The authors declare that there is no conflict of interests regarding the publication of this paper.

%\section*{Declaration of Generative AI and AI-assisted technologies in the writing process}

\section*{Funding}
Hidetaka Hamada is partially supported by
JSPS KAKENHI Grant Number JP22K03363.

\section*{Data Availability Statement}
Data sharing not applicable to this article as no datasets were generated or analysed during the current study.

\section*{Declartion}

\section*{Conflict of interest}
The authors have no relevant financial or non-financial interests
to disclose.

%\section*{Author Contributions}
%All authors have contributed to all aspects of this manuscript and have reviewed its final draft.
%%%%%%%%%%%%%%%%%%%%%%%%%%%%%%%%%%%%%%%%%%%%%%%%%%%%%%%%%%%%%%%%%%%%%%%%%%%%%
%All authors contributed to the study conception and design.
%The first draft of the manuscript was written by Shaolin Chen and all authors commented
%on previous versions of the manuscript. All authors read and approved the final manuscript.
%%%%%%%%%%%%%%%%%%%%%%%%%%%%%%%%%%%%%%%%%%%%%%%%%%%%%%%%
%Conceptualization, H.H., G.K. and M.K.;
%methodology, H.H., G.K. and M.K.;
%software, H.H.;
%validation,H.H., G.K. and M.K.;
%formal analysis, H.H., G.K. and M.K.;
%investigation, H.H., G.K. and M.K.;
%resources, H.H.;
%data curation, H.H.;
%writing---original draft preparation, H.H.;
%writing---review and editing, H.H. and M.K.;
%visualization, H.H.;
%supervision, H.H.;
%project administration, H.H.;
%funding acquisition, H.H.;
%H.H. and M.K. All authors have read and agreed to the published version of the manuscript.

%\section*{Acknowledgments}
%\label{acknowledgments}

%Hidetaka Hamada is partially supported by JSPS KAKENHI Grant
%Number JP16K05217.
%%%%%%%%%%%%%%%%%%%%%%%%%%%%%%%%%%%%%%%%%%%%%

%%%%%%%%%%%%%%%%%%%%%%%%%%%%%%%%%%%%%%%

%\bibliographystyle{amsplain}

\end{document}